\documentclass[10pt, a4paper]{article}%
\usepackage{amsmath,amssymb,eucal}
\usepackage{amsfonts}
\usepackage{mathrsfs}
\usepackage{slashed}
\usepackage{graphicx}
\usepackage{ifpdf}
 \usepackage{url}
\numberwithin{equation}{section} 
\def\beq{\begin{equation}}
\def\eeq{\end{equation}}

\def\bR{ {{\mathbb{R}}}}

\def\Tr{ {{\rm{Tr}}} }

\newcommand{\sgn}{{\rm sgn}}

\newcommand{\pk}[1]{p_{\kappa}}

\newcommand{\ol}{\overline{l}}
\newcommand{\ul}{\underline{l}}
\newcommand{\oL}{\overline{L}}
\newcommand{\uL}{\underline{L}}
\newcommand{\on}{\overline{n}}
\newcommand{\un}{\underline{n}}

\newtheorem{defn}{{\bf Definition}}[section]
\newtheorem{thm}[defn]{{\bf Theorem}}
\newtheorem{cor}[defn]{{\bf Corollary}}
\newtheorem{lem}[defn]{{\bf Lemma}}

\newtheorem{rem}[defn]{{\bf Remark}}

\newtheorem{notation}[defn]{Notation}
\newenvironment{proof}[1][Proof]{\textbf{#1.} }{\hfill \rule{0.5em}{0.5em}}
\begin{document}

\title{Path Integral Quantization of Volume}
\author{Adrian P. C. Lim \\
Email: ppcube@gmail.com
}

\date{}

\maketitle

\begin{abstract}
A hyperlink is a finite set of non-intersecting simple closed curves in $\mathbb{R} \times \mathbb{R}^3$. Let $R$ be a compact set inside $\mathbf{R}^3$. The dynamical variables in General Relativity are the vierbein $e$ and a $\mathfrak{su}(2)\times\mathfrak{su}(2)$-valued connection $\omega$. Together with Minkowski metric, $e$ will define a metric $g$ on the manifold. Denote $V_R(e)$ as the volume of $R$, for a given choice of $e$.

The Einstein-Hilbert action $S(e,\omega)$ is defined on $e$ and $\omega$. We will quantize the volume of $R$ by integrating $V_R(e)$ against a holonomy operator of a hyperlink $L$, disjoint from $R$, and the exponential of the Einstein-Hilbert action, over the space of vierbein $e$ and $\mathfrak{su}(2)\times\mathfrak{su}(2)$-valued connection $\omega$. Using our earlier work done on Chern-Simons path integrals in $\mathbb{R}^3$, we will write this infinite dimensional path integral as the limit of a sequence of Chern-Simons integrals. Our main result shows that the volume operator can be computed by counting the number of half-twists in the projected hyperlink, which lie inside $R$. By assigning an irreducible representation of $\mathfrak{su}(2)\times\mathfrak{su}(2)$ to each component of $L$, the volume operator gives the total kinetic energy, which comes from translational and angular momentum.
\end{abstract}

\hspace{.35cm}{\small {\bf MSC} 2010: } 83C45, 81S40, 81T45, 57R56 \\
\indent \hspace{.35cm}{\small {\bf Keywords}: Volume, Path integral, Einstein-Hilbert, Loop representation, Quantum gravity}




\section{Quantization of volume}

Spin networks are used by the authors in \cite{Baez:1999sr}, \cite{Baez:1999:Online}, \cite{Baez1996253} and \cite{Rovelli:1995ac} to develop quantum gravity. A spin network in $\bR^3$ is essentially a graph, each vertex has valency 3 and a spin is assigned to each edge, satisfying a certain inequality for edges incident on a common vertex. For $\bR^4$, one has to use spin foam, which we refer the reader to \cite{Baez:1999sr} and \cite{Baez:1999:Online}.

In quantum gravity, the underlying metric is a dynamical variable. In the quantization of gravity, one of the classical property that one wants to quantize is volume of a three dimensional region $R$, which can be defined as a functional of this metric. To quantize the volume, one would attempt to apply canonical quantization and promote the metric to be an operator, which is no easy task.

One way to overcome this problem would be to use Ashtekar variables. In \cite{rovelli1995discreteness}, the authors wrote the volume of a three dimensional region $R$ in terms of these Ashtekar variables and used canonical quantization on Ashtekar variables to define a volume operator. The result they obtained was to count the vertices in a graph (spin network) which lie in a region $R$, weighted by some term related to the spin, hence the eigenvalues of the operator can be explicitly calculated. Their derivation however, is not mathematically rigorous.

We will instead quantize the volume of $R$, using a Einstein-Hilbert path integral approach, which we will refer the reader to \cite{EH-Lim02} for details. The idea actually comes from \cite{PhysRevLett.61.1155}, which used loops in $\bR^4$ to describe quantum gravity. More importantly, the authors wrote down a path or functional integral using a suitable (infinite dimensional) measure. To understand our main result, we will first need to describe loops in $\bR^4$.

Consider $\bR^4 \equiv \bR \times \bR^3$, whereby $\bR$ will be referred to as the time-axis and $\bR^3$ is the spatial 3-dimensional Euclidean space. In future, when we write $\bR^3$, we refer to the spatial subspace in $\bR^4$. Let $\pi_0: \bR^4 \rightarrow \bR^3$ denote this projection. Fix the standard coordinates on $\bR^4\equiv \bR \times \bR^3 $, with time coordinate $x_0$ and spatial coordinates $(x_1, x_2, x_3)$.

Let $\{e_i\}_{i=1}^3$ be the standard basis in $\bR^3$. And $\Sigma_i$ is the plane in $\bR^3$, containing the origin, whose normal is given by $e_i$. So, $\Sigma_1$ is the $x_2-x_3$ plane, $\Sigma_2$ is the $x_3-x_1$ plane and finally $\Sigma_3$ is the $x_1-x_2$ plane. Let $\pi_i: \bR^4 \rightarrow \bR \times \Sigma_i$ denote this projection.

For a finite set of non-intersecting simple closed curves in $\bR^3$ or in $\bR \times \Sigma_i$, we will refer to it as a link. If it has only one component, then this link will be referred to as a knot. A simple closed curve in $\bR^4$ will be referred to as a loop. A finite set of non-intersecting loops in $\bR^4$ will be referred to as a hyperlink in this article. We say a link or hyperlink is oriented if we assign an orientation to its components.

Let $L$ be a hyperlink. We say $L$ is a time-like hyperlink, if given any 2 distinct points $p\equiv (x_0, x_1, x_2, x_3), q\equiv (y_0, y_1, y_2, y_3) \in L$, $p \neq q$, we have
\begin{itemize}
  \item $\sum_{i=1}^3(x_i - y_i)^2 > 0$;
  \item if there exists $i, j$, $i \neq j$ such that $x_i = y_i$ and $x_j = y_j$, then $x_0 - y_0 \neq 0$.
\end{itemize}

Throughout this article, all our hyperlinks in consideration will be time-like. A time-like hyperlink will imply that $\pi_a(L)$, $a=0, 1, 2, 3$, are all links inside their respective 3-dimensional subspace $\pi_a(\bR^4)\subset \bR^4$. See \cite{EH-Lim06}.

We adopt Einstein's summation convention, i.e. we sum over repeated superscripts and subscripts. Indices such as $a,b,c, d$ and greek indices such as $\mu,\gamma, \alpha, \beta$ will take values from 0 to 3; indices labeled $i, j, k$, $\bar{i}, \bar{j}, \bar{k}$ will only take values from 1 to 3. 

\section{Coloring of matter hyperlink}\label{s.cmh}

Let $\mathfrak{su}(2)$ be the Lie Algebra of $SU(2)$. We need to first choose a basis for $\mathfrak{su}(2)$ and we shall make the following choice,
\beq \breve{e}_1 :=
\frac{1}{2}\left(
  \begin{array}{cc}
    0 &\ 1 \\
    -1 &\ 0 \\
  \end{array}
\right),\ \ \breve{e}_2 :=
\frac{1}{2}\left(
  \begin{array}{cc}
    0 &\ i \\
    i &\ 0 \\
  \end{array}
\right),\ \ \breve{e}_3 :=
\frac{1}{2}\left(
  \begin{array}{cc}
    i &\ 0 \\
    0 &\ -i \\
  \end{array}
\right). \nonumber \eeq

Choose the group $SU(2) \times SU(2)$, which define a spin structure on $\bR^4$. And define the following basis in $\mathfrak{su}(2) \times \mathfrak{su}(2)$,
\begin{align*}
\hat{E}^{01} = (\breve{e}_1, 0),\ \ \hat{E}^{02} = (\breve{e}_2, 0), \ \ \hat{E}^{03} = (\breve{e}_3, 0), \\
\hat{E}^{23} = (0, \breve{e}_1),\ \ \hat{E}^{31} = ( 0, \breve{e}_2), \ \ \hat{E}^{12} = (0, \breve{e}_3),
\end{align*}
and write
\beq \hat{E}^{\tau(1)} = \hat{E}^{23}, \ \ \hat{E}^{\tau(2)} = \hat{E}^{31}, \ \ \hat{E}^{\tau(3)} = \hat{E}^{12}. \nonumber \eeq We will also write $\hat{E}^{\alpha\beta} = -\hat{E}^{\beta\alpha}$.

Using the above basis, we will also define \beq \mathcal{E}^\pm := \sum_{i=1}^3\breve{e}_i \in \mathfrak{su}(2). \label{e.sux.1} \eeq

Let $\rho^\pm: \mathfrak{su}(2) \rightarrow {\rm End}(V^\pm)$ be an irreducible finite dimensional representation, indexed by half-integer and integer values $j_{\rho^\pm} \geq 0$. Without loss of generality, we assume that $\rho^\pm(\hat{E})$ is skew-Hermitian for any $\hat{E} \in \mathfrak{su}(2)$.

The representation $\rho: \mathfrak{su}(2) \times \mathfrak{su}(2) \rightarrow {\rm End}(V^+) \times {\rm End}(V^-)$ will be given by $\rho = (\rho^+, \rho^-)$, with \beq \rho: \alpha_i\hat{E}^{0i} + \beta_j \hat{E}^{\tau(j)} \mapsto \left(\sum_{i=1}^3\alpha_i \rho^+(\breve{e}_i) , \sum_{j=1}^3\beta_j \rho^-(\breve{e}_j) \right). \nonumber \eeq By abuse of notation, we will now write $\rho^+ \equiv (\rho^+, 0)$ and $\rho^- \equiv (0, \rho^-)$ in future and thus $\rho^+(\hat{E}^{0i}) \equiv \rho^+(\breve{e}_i)$, $\rho^-(\hat{E}^{\tau(j)}) \equiv \rho^-(\breve{e}_j)$.

Consider 2 different hyperlinks, $\oL = \{\ol^u:\ u=1, \ldots, \on\}$ and $\uL = \{\ul^v:\ v=1, \ldots, \un\}$. The former will be called a matter hyperlink; the latter will be referred to as a geometric hyperlink. The symbols $u, \bar{u}, v, \bar{v}$ will be indices, taking values in $\mathbb{N}$. They will keep track of the loops in our hyperlinks $\oL$ and $\uL$. The symbols $\on$ and $\un$ will always refer to the number of components in $\oL$ and $\uL$ respectively.

Color the matter hyperlink, which means choose a representation $\rho_u: \mathfrak{su}(2) \times \mathfrak{su}(2) \rightarrow {\rm End}(V_u^+) \times {\rm End}(V_u^-)$ for each component $\ol^u$, $u=1, \ldots, \on$, in the hyperlink $\oL$. Note that we do not color $\uL$, i.e. we do not choose a representation for $\uL$.

In $\mathfrak{su}(2) \times \mathfrak{su}(2)$, the first copy of $\mathfrak{su}(2)$ is generated by $\{\breve{e}_i\}_{i=1}^3$, which corresponds to boost in the $x_i$ direction in the Lorentz group; the second copy of $\mathfrak{su}(2)$ is generated by another independent set $\{\breve{e}_i\}_{i=1}^3$, which corresponds to rotation about the $x_i$-axis in the Lorentz group. When we give a representation $\rho^\pm$ to a colored loop $\ol$, which we interpret as representing a particle, we are effectively assigning values to the translational and angular momentum of this particle.

Each irreducible representation $\rho^\pm$ will define the following Casimir operator,
\begin{align*}
\sum_{i=1}^3 \rho^+(\hat{E}^{0i})\rho^+(\hat{E}^{0i})  = -\xi_{\rho^+} I_{\rho^+},\\
\sum_{i=1}^3 \rho^-(\hat{E}^{\tau(i)})\rho^-(\hat{E}^{\tau(i)}) = -\xi_{\rho^-} I_{\rho^-},
\end{align*}
$I_{\rho^\pm}$ is the $2j_{\rho^\pm} + 1$ identity operator for $V^\pm$ and $\xi_{\rho^\pm} := j_{\rho^\pm}(j_{\rho^\pm}+1)$. Note that the dimension of $V^\pm$ is $2j_{\rho^\pm}+1$. We can interpret $\xi_{\rho^+}$ to be the kinetic energy arising from boost; $\xi_{\rho^-}$ to be the kinetic energy arising from rotation.

Given a colored hyperlink $\oL$ and a hyperlink $\uL$, we also assume that together (by using ambient isotopy if necessary), they form another hyperlink with $\on + \un$ components. Denote this new colored hyperlink by $\chi(\oL, \uL) \equiv \chi(\{\ol^u\}_{u=1}^{\on}, \{\ul^v\}_{v=1}^{\un})$, assumed to be time-like.

Consider an oriented hyperlink $\chi(\ol, \ul)$, made up of 2 distinct oriented loops $\ol$ and $\ul$. In \cite{EH-Lim03} or \cite{EH-Lim06}, we defined the hyperlinking number between $\ol$ and $\ul$ in $\bR \times\bR^3$, denoted as ${\rm sk}(\ol, \ul)$, to distinguish from the linking number between 2 simple closed curves in $\bR^3$. More generally, define for each $u=1, \ldots, \on$, \beq {\rm sk}(\ol^u, \uL):= \sum_{v=1}^{\un}{\rm sk}(\ol^u, \ul^v), \nonumber\eeq calculated from $\chi(\oL, \uL)$.

\section{Volume path integral}

Let $\overline{\mathcal{S}}_\kappa(\bR^4) \subset L^2(\bR^4)$ be a Schwartz space, as defined in \cite{EH-Lim02}. Using the standard coordinates on $\bR^4$, let $\Lambda^1(\bR^3)$ denote the subspace in $\Lambda^1(\bR^4)$ spanned by $\{dx_1, dx_2, dx_3\}$. Define
\begin{align*}
L_\omega :=& \overline{\mathcal{S}}_\kappa(\bR^4) \otimes \Lambda^1(\bR^3)\otimes \mathfrak{su}(2) \times \mathfrak{su}(2), \\
L_e :=& \overline{\mathcal{S}}_\kappa(\bR^4) \otimes \Lambda^1(\bR^3)\otimes V,
\end{align*}
whereby $\bR^4 \times V \rightarrow \bR^4$ is a trivial 4-dimensional vector bundle, with structure group $SO(3,1)$. This implies that $V$ is endowed with a Minkowski metric, $\eta^{ab}$, of signature $(-, +, +, +)$. Let $\{E^\gamma\}_{\gamma=0}^3$ be a basis for $V$.

Given $\omega \in L_\omega$ and $e \in L_e$, we will write
\begin{align*}
\omega =& A^i_{\alpha\beta} \otimes dx_i\otimes \hat{E}^{\alpha\beta} \in \overline{\mathcal{S}}_\kappa(\bR^4) \otimes \Lambda^1(\bR^3)\otimes \mathfrak{su}(2) \times \mathfrak{su}(2), \\
e =& B^i_\gamma \otimes dx_i\otimes E^\gamma \in \overline{\mathcal{S}}_\kappa(\bR^4) \otimes \Lambda^1(\bR^3) \otimes V.
\end{align*}
There is an implied sum over repeated indices.

\begin{rem}
Note that $A^{i}_{\alpha\beta} = -A^{i}_{\beta\alpha}\in \overline{\mathcal{S}}_\kappa(\bR^4)$.
\end{rem}

In \cite{EH-Lim02}, we define the Einstein-Hilbert action, after applying axial gauge fixing, as ($\partial_0 \equiv \partial/\partial x_0$)
\begin{align*}
S_{EH}(e, \omega) :=
\frac{1}{8}&\int_{\bR^4}\epsilon^{abcd}B^1_\gamma B^2_\mu[E^{\gamma \mu}]_{ab} \cdot \partial_0 A^3_{\alpha\beta}[E^{\alpha\beta}]_{cd} dx_1\wedge dx_2 \wedge dx_0 \wedge dx_3\\
+& \frac{1}{8}\int_{\bR^4}\epsilon^{abcd}B^2_\gamma B^3_\mu[E^{\gamma \mu}]_{ab} \cdot \partial_0 A^1_{\alpha\beta}[E^{\alpha\beta}]_{cd} dx_2\wedge dx_3 \wedge dx_0 \wedge dx_1\\
+&\frac{1}{8}\int_{\bR^4}\epsilon^{abcd}B^3_\gamma B^1_\mu[E^{\gamma \mu}]_{ab} \cdot \partial_0 A^2_{\alpha\beta}[E^{\alpha\beta}]_{cd} dx_3\wedge dx_1 \wedge dx_0 \wedge dx_2.
\end{align*}
We sum over repeated indices and $\epsilon^{\mu \gamma \alpha \beta} \equiv \epsilon_{\mu \gamma \alpha \beta}$ is equal to 1 if the number of transpositions required to permute $(0123)$ to $(\mu\gamma\alpha\beta)$ is even; otherwise it takes the value -1.


Let $q \in \bR$ be known as a charge. Define
\begin{align*}
V(\{\ul^v\}_{v=1}^{\underline{n}})(e) :=& \exp\left[ \sum_{v=1}^{\un} \int_{\ul^v} \sum_{\gamma=0}^3 B^i_\gamma \otimes dx_i\right], \\
W(q; \{\ol^u, \rho_u\}_{u=1}^{\on})(\omega) :=& \prod_{u=1}^{\on}\Tr_{\rho_u}  \mathcal{T} \exp\left[ q\int_{\ol^u} A^i_{\alpha\beta} \otimes dx_i\otimes \hat{E}^{\alpha\beta}  \right].
\end{align*}
Here, $\mathcal{T}$ is the time-ordering operator as defined in \cite{CS-Lim02}. And we sum over repeated indices, with $i$ taking values in 1, 2 and 3; $\alpha\neq \beta$ take values in 0, 1, 2, 3 and $B^i_\gamma, A^i_{\alpha\beta} \in \overline{\mathcal{S}}_\kappa(\bR^4)$.

\begin{rem}
The term $\mathcal{T} \exp\left[ q\int_{\ol^u} \omega \right]$ is known as a holonomy operator along a loop $\ol^u$, for a spin connection $\omega$.
\end{rem}

Let $\vec{y}^u \equiv (y_0^u, y_1^u, y_2^u, y_3^u) : I := [0,1] \rightarrow \bR \times \bR^3$ be a parametrization of a loop $\ol^u \subset \oL$, $u=1, \ldots, \on$. We will write $y^u(s) = (y_1^u(s), y_2^u(s), y_3^u(s))$ and $\vec{y}^u(s) \equiv \vec{y}_s^u$. We will also write $\vec{y}^u = (y^u_0, y^u)$. Similarly, choose a parametrization $\vec{\varrho}^{v}: I \rightarrow \bR \times \bR^3$, $v = 1, \ldots, \un$, for each loop $\ul^v \subset \uL$. When the loop is oriented, we will often choose a parametrization which is consistent with the assigned orientation.

Fix a closed and bounded 3-manifold $R \subset \bR^3 \cong \{0\} \times \bR^3$, possibly disconnected with finite number of components. Henceforth, we will refer to $R$ as a compact region. We further assume that $\oL$ is disjoint from $R$.

Now we will proceed to quantize the volume. Using the dynamical variables $\{B_\mu^i\}$ and the Minkowski metric $\eta^{ab}$, we see that the metric $g^{ab} \equiv B^a_\mu\eta^{\mu\gamma}B^b_\gamma$ and the volume $V_R$ is given by \beq V_R(e) := \int_R \sqrt{\epsilon_{ijk}\epsilon_{\bar{i}\bar{j}\bar{k}}g^{i\bar{i}}g^{j\bar{j}}g^{k\bar{k}}}. \nonumber \eeq

Consider the following path integral, \beq \frac{1}{Z}\int_{\omega \in L_\omega,\ e \in L_e}V_R(e)V(\{\ul^v\}_{v=1}^{\un})(e) W(q; \{\ol^u, \rho_u\}_{u=1}^{\on})(\omega)\  e^{i S_{EH}(e, \omega)}\ De D\omega, \label{e.v.11} \eeq whereby $De$ and $D\omega$ are Lebesgue measures on $L_e$ and $L_\omega$ respectively and \beq Z = \int_{\omega \in L_\omega,\ e \in L_e}e^{i S_{EH}(e, \omega)}\ De D\omega. \label{e.ehv.1} \eeq

\begin{rem}\label{r.z.1}
\begin{enumerate}
  \item When $R$ is the empty set, we define $V_\emptyset \equiv 1$, so we write Expression \ref{e.v.11} as $Z(q; \chi(\oL, \uL))$, which in future be termed as the Wilson Loop observable of the colored hyperlink $\chi(\oL, \uL))$.
  \item The volume operator will henceforth be denoted by $\hat{V}_R$ and we will write Expression \ref{e.v.11} as $\hat{V}_R[Z(q; \chi(\oL, \uL))]$.
\end{enumerate}
\end{rem}

\begin{notation}\label{n.r.2}
Let $\rho: I^3 \rightarrow \bR^3$ be any parametrization of a compact region $R \subset \bR^3$. Let $|J_\rho|(r)$ denote the determinant of the Jacobian of $\rho\equiv (\rho_1, \rho_2, \rho_3)$, $r = (r_1, r_2, r_3)$. And write $dr = dr_1 dr_2 dr_3$. We will also write $\vec{\rho}(r) \equiv \vec{\rho}_r \equiv (0, \rho(r)) \in \bR^4$.

Let $L = \{\ol^1, \ldots, \ol^{\on}\}$ be a matter hyperlink. Project each $\ol^u$ into $\bR^3$ to form a knot $l^u$ and let $\mathcal{N}(l^u)$ be a tubular neighborhood of $l^u$. Write $R = \bigcup_{v=1}^{\bar{m}}R_v$ as a disjoint union, such that either $R_v \subseteq \mathcal{N}(l^u)$ for some $u$ or $R_v \cap \mathcal{N}(l^u) = \emptyset$ for every $u = 1, \ldots, \on$. Let $I_v^3 \subset I^3$ such that $\rho: I_v^3 \rightarrow R_v \subset R$ be a parametrization of $R_v$.
\end{notation}

In \cite{EH-Lim02}, we showed that we can define Expression \ref{e.v.11} and write it as the limit as $\kappa$ goes to infinity, of the following expression
\begin{align}
q^2\prod_{\bar{u}=1}^{\on}\Bigg\{&  \Bigg[ \sum_{v=1}^{\bar{m}}\tilde{\kappa}\sum_{u=1}^{\on} \int_{r\in I_v^3}dr |J_\rho|(r)\Bigg|\int_{I^2} d\hat{s}\  \epsilon^{ijk}\left\langle p_\kappa^{\vec{y}_{s}^u}, p_\kappa^{\vec{\rho}(r)} \right\rangle_k y_{i,s}^{u,\prime}y_{j,\bar{s}}^{u,\prime} \nonumber \\
&\hspace{2cm}\times e^{-\kappa^2|y_{\bar{s}}^{u} - \rho(r)|^2/8} \left\langle \partial_0^{-1}q_\kappa^{y_{0,\bar{s}}^{u}}, q_\kappa^{0} \right\rangle
\xi_{\rho_{u}^+}
\Bigg| \Bigg]^{1/\on} \Tr_{\rho_{\bar{u}}^+}\hat{\mathcal{W}}^+_\kappa(q; \ol^{\bar{u}}, \uL) \nonumber\\
+&\Bigg[ \sum_{v=1}^{\bar{m}} \tilde{\kappa}\sum_{u=1}^{\on} \int_{r\in I_v^3}dr |J_\rho|(r)\Bigg|\int_{I^2} d\hat{s}\  \epsilon^{ijk}\left\langle p_\kappa^{\vec{y}_{s}^u}, p_\kappa^{\vec{\rho}(r)} \right\rangle_k y_{i,s}^{u,\prime}y_{j,\bar{s}}^{u,\prime} \nonumber \\
&\hspace{2cm}\times e^{-\kappa^2|y_{\bar{s}}^{u} - \rho(r)|^2/8} \left\langle \partial_0^{-1}q_\kappa^{y_{0,\bar{s}}^{u}}, q_\kappa^{0} \right\rangle
\xi_{\rho_{u}^-}
\Bigg| \Bigg]^{1/\on} \Tr_{\rho_{\bar{u}}^-}\hat{\mathcal{W}}^-_\kappa(q; \ol^{\bar{u}}, \uL)
\Bigg\}, \label{ex.v.3}
\end{align}
the notations used in the expression will be explained in Section \ref{s.vo}. Note that \beq \tilde{\kappa} = \frac{\sqrt\pi}{2}\frac{\kappa}{4}\left(\frac{\kappa}{\sqrt{2\pi}}\right)^2
\left(\frac{\kappa^{2} }{8\pi}\right)^2. \nonumber \eeq See Definition Volume Path Integral in \cite{EH-Lim02}.

\begin{rem}
Expression \ref{ex.v.3} will of course depend on the choice of partition $\{R_v\}_{v=1}^{\bar{m}}$. But its limit as $\kappa$ goes to infinity will be shown to be independent of this partition in a sequel.
\end{rem}

A framed hyperlink $L$ is a hyperlink with a projected framed link $\pi_0(L)$. A framed link has a frame defined on it, which results in the addition of half-twists to each component graph in the link diagram. When we project a framed link on a plane as in Definition 2.6 in \cite{CS-Lim02}, we obtain a graph, each vertex has valency 2 or 4. Each vertex with valency 2 represents a half-twist. We can define an algebraic crossing number for a half-twist.

Let $\tilde{\pi}_i: \bR^3 \rightarrow \Sigma_i$. Our main Theorem \ref{t.main.3} says that the volume operator is computed by projecting $\pi_0(\oL)$ and $R$ on a plane $\Sigma_i$ and counting all the half-twists on the graphs, which are in the interior of the planar set $\tilde{\pi}_i(R)$, weighted by the kinetic energy. The kinetic energy is derived from momentum coming from boosts and also from angular momentum, which are given by $\xi_{\rho^+}$ and $\xi_{\rho^-}$ respectively. Thus the volume operator measures the total kinetic energy of a set of particles, represented by the hyperlink $\oL$. 


It was remarked in \cite{rovelli1995discreteness} that the eigenvalues for the volume operator come from matter in the ambient space. This is consistent with our main result, whereby the eigenvalues are computed from the matter hyperlink $\oL$. In fact, the author in \cite{Mercuri:2010xz} talked about a discrete structure in space-time. Only nodes on a spin network will contribute to the volume, which the author in \cite{rovelli2004quantum} interpret as an ensemble of quanta of volume. This agrees with our computations, whereby half-twists from a link contribute to these `chunks' of space. The only important difference, is that our result interprets the quantum eigenvalues as kinetic energy. As an application, we will explain how quantum gravity solves certain inconsistencies in General Relativity as outlined in \cite{Thiemann:2002nj}.

\section{Volume operator}\label{s.vo}

Let us now explain the notations in Expression \ref{ex.v.3}.

In this article, $\vec{y} \equiv (y_0, y) \in \bR^4$, whereby $y \equiv (y_1, y_2, y_3) \in \bR^3$. We will write
\beq \hat{y}_i =
 \left\{
  \begin{array}{ll}
    (y_2, y_3), & \hbox{$i=1$;} \\
    (y_1, y_3), & \hbox{$i=2$;} \\
    (y_1, y_2), & \hbox{$i=3$.}
  \end{array}
\right. \nonumber \eeq

If $x \in \bR^n$, we will write $(p_\kappa^x)^2$ to denote the $n$-dimensional Gaussian function, center at $x$, variance $1/\kappa^2$. For example, \beq p_\kappa^x(\cdot) = \frac{\kappa^2}{2\pi}e^{-\kappa^2|\cdot - x|^2/4},\ x \in \bR^4. \nonumber \eeq We will also write $(q_\kappa^x)^2$ to denote the 1-dimensional Gaussian function, i.e. \beq q_\kappa^x(\cdot) = \frac{\sqrt{\kappa}}{(2\pi)^{1/4}}e^{-\kappa^2 (\cdot - x)^2/4}. \nonumber \eeq

For $x, y \in \bR^2$, we write \beq \langle p_\kappa^x, p_\kappa^y \rangle = \int_{z\in \bR^2} \frac{\kappa}{\sqrt{2\pi}}e^{-\kappa^2|z - x|^2/4} \frac{\kappa}{\sqrt{2\pi}}e^{-\kappa^2|z - y|^2/4} dz, \nonumber \eeq i.e. we integrate over Lebesgue measure on $\bR^2$. 

For $x = (x_0,x_1, x_2, x_3)$, write \beq x(s_a) :=
\left\{
  \begin{array}{ll}
    (s_0,x_1, x_2, x_3), & \hbox{$a=0$;} \\
    (x_0,s_1, x_2, x_3), & \hbox{$a=1$;} \\
    (x_0,x_1, s_2, x_3), & \hbox{$a=2$;} \\
    (x_0,x_1, x_2, s_3), & \hbox{$a=3$.}
  \end{array}
\right. \nonumber \eeq

Let $\partial_a \equiv \partial/\partial x_a$ be a differential operator. There is an operator $\partial_a^{-1}$ acting on a dense subset in $\overline{\mathcal{S}}_\kappa(\bR^4)$, \beq (\partial_a^{-1}f)(x) := \frac{1}{2}\int_{-\infty}^{x_a} f(x(s_a))\ ds_a - \frac{1}{2}\int_{x_a}^{\infty} f(x(s_a))\ ds_a,\ f \in \overline{\mathcal{S}}_\kappa(\bR^4). \label{e.d.1} \eeq Here, $x_a \in \bR$. Notice that $\partial_a\partial_a^{-1}f \equiv f$ and $\partial_a^{-1}f$ is well-defined provided $f$ is in $L^1$.

For each $i = 1, 2, 3$, write
\beq \left\langle p_\kappa^{\vec{x}}, p_\kappa^{\vec{y}} \right\rangle_i :=
\left\langle p_\kappa^{\hat{x}_{i}}, p_\kappa^{\hat{y}_{i}} \right\rangle \left\langle q_\kappa^{x_{i}}, \kappa\partial_0^{-1}q_\kappa^{y_{i}} \right\rangle\left\langle \partial_0^{-1}q_\kappa^{x_{0}}, q_\kappa^{y_{0}}\right\rangle. \nonumber \eeq

Here, \beq \partial_0^{-1}q_\kappa^{x_{0}}(t) \equiv \frac{1}{2}\int_{-\infty}^t q_\kappa^{x_{0}}(\tau)\ d\tau -
\frac{1}{2}\int_{t}^\infty q_\kappa^{x_{0}}(\tau)\ d\tau. \nonumber \eeq

Note that $\left\langle \partial_0^{-1}q_\kappa^{x_{0}}, q_\kappa^{y_{0}}\right\rangle \equiv \left\langle q_\kappa^{y_{0}}, \partial_0^{-1}q_\kappa^{x_{0}} \right\rangle$  means we integrate $\partial_0^{-1}q_\kappa^{x_{0}} \cdot q_\kappa^{y_{0}}$ over $\bR$, using Lebesgue measure. It is well-defined because $q_\kappa^{x_0}$ is in $L^1$.

Recall we parametrize $\ol^u$ and $\ul^v$ using $\vec{y}^u$ and $\vec{\varrho}^v$ respectively, $u=1, \ldots \on$, $v=1, \ldots \un$. Define $\hat{\mathcal{W}}_\kappa^\pm(q; \ol^u, \uL)$ as
\begin{align}
\hat{\mathcal{W}}_\kappa^\pm(q; \ol^u, \uL) := \exp\left[ \mp\frac{iq}{4}\frac{\kappa^3}{4\pi}\sum_{v=1}^{\underline{n}}\int_{I^2}\ d\hat{s}\ \epsilon^{ijk}\left\langle  p_\kappa^{\vec{y}^u_s}, p_\kappa^{\vec{\varrho}^v_{\bar{s}}}\right\rangle_k  y^{u,\prime}_{i,s}\varrho^{v,\prime}_{j,\bar{s}} \otimes \mathcal{E}^\pm\right], \label{e.h.5}
\end{align}
And $\epsilon^{ijk} \equiv \epsilon_{ijk}$ be defined on the set $\{1,2,3\}$, by \beq \epsilon^{123} = \epsilon^{231} = \epsilon^{312} = 1,\ \ \epsilon^{213} = \epsilon^{321} = \epsilon^{132} = -1, \nonumber \eeq if $i,j,k$ are all distinct; 0 otherwise.

Using the lemma in the appendix found in \cite{EH-Lim02}, one can show that \beq  \lim_{\kappa \rightarrow \infty}\Tr_{\rho_u^\pm}\ \hat{\mathcal{W}}_\kappa^\pm(q; \ol^u, \uL) = \Tr_{\rho^\pm_u}\ \exp[\mp\pi iq\ {\rm sk}(\ol^u, \uL) \cdot \mathcal{E}^\pm], \label{e.w.5}  \eeq whereby $\mathcal{E}^\pm$ was defined in Equation (\ref{e.sux.1}).
For a detailed proof, the reader can refer to \cite{EH-Lim03}.

In Remark \ref{r.z.1} for the special case when $R = \emptyset$, we can define the Wilson Loop observable for a colored hyperlink $\chi(\oL, \uL)$,
\begin{align*}
Z(q; \chi(\oL, \uL) )
:=& \lim_{\kappa \rightarrow \infty}\prod_{u=1}^{\on}\left[\Tr_{\rho_u^+}\ \hat{\mathcal{W}}_\kappa^+(q; \ol^u, \uL) + \Tr_{\rho_u^-}\ \hat{\mathcal{W}}_\kappa^-(q; \ol^u, \uL) \right]\\
=& \prod_{u=1}^{\on}\left( \Tr_{\rho^+_u}\ \exp[\pi iq\ {\rm sk}(\ol^u, \uL) \cdot \mathcal{E}^+] +
\Tr_{\rho^-_u}\ \exp[-\pi iq\ {\rm sk}(\ol^u, \uL) \cdot \mathcal{E}^-] \right).
\end{align*}
Note that $\Tr_{\rho^\pm_u}$ means take the trace. See \cite{EH-Lim02}.

Let $p_\theta(x,y) = \frac{1}{ 2\pi \theta}e^{-|x-y|^2/2\theta}$ be the 2-dimensional Gaussian function. Recall it means the transition probability of being in position $y \in \bR^2$, given that our last known position is in $x \in \bR^2$. Using Item 2 in the lemma in the appendix found in \cite{EH-Lim02}, we have
\begin{align}
p_\theta(x,y) =& \int_{z \in \bR^2} p_{\theta/2}(x,z) p_{\theta/2}(z,y)\ dz \nonumber \\
=& \int_{z\in \bR^2}
\frac{1}{2\pi\theta/2}e^{-|x-z|^2/\theta} \cdot \frac{1}{2\pi\theta/2}e^{-|z-y|^2/\theta}\ dz. \label{e.vc.1}
\end{align}

Let $V \subset \bR^3$ be a compact region which contains $x$ and $y$. Then clearly, as $\theta \rightarrow 0$, we have \beq \int_{z\in V^c} \frac{1}{\theta^{3/2}}\frac{2^{3/2}}{(\sqrt{2\pi})^3}e^{-|x-z|^2/\theta} \cdot \frac{1}{\theta^{3/2}}\frac{2^{3/2}}{(\sqrt{2\pi})^3}e^{-|z-y|^2/\theta}\ dz\ \longrightarrow 0. \label{e.vc.5}\eeq

\begin{lem}\label{l.l.6}
Let $s \neq t$ and let $\sgn(t-s)$ denote the sign of $t-s$. We have \beq \frac{\kappa^2}{4\pi}\int_{\bR}\langle \partial_0^{-1}q_\kappa^{s}, q_\kappa^{z}\rangle e^{-\kappa^2|t - z|^2/8} dz \rightarrow \sgn(t-s), \nonumber \eeq as $\kappa \rightarrow \infty$.
\end{lem}

\begin{proof}
Without loss of generality, we assume that $t > s$. Now
\begin{align}
\Big| \frac{\kappa^2}{4\pi}&\int_{\bR}\langle \partial_0^{-1}q_\kappa^{s}, q_\kappa^{z}\rangle e^{-\kappa^2|t - z|^2/8}\ dz -1 \Big| \nonumber\\
\leq& \frac{\kappa}{2\sqrt{2\pi}}\int_{\bR}\left|\frac{\kappa}{\sqrt{2\pi}}\langle \partial_0^{-1}q_\kappa^{s}, q_\kappa^{z}\rangle
-1 \right| \ e^{-\kappa^2|t - z|^2/8}\ dz .\label{e.x.1}
\end{align}
Choose an $\epsilon > 0$ such that $t - \epsilon > s$. Then,
\begin{align*}
\frac{\kappa}{2\sqrt{2\pi}}\int_{(t-\epsilon, t + \epsilon)}\left|\frac{\kappa}{\sqrt{2\pi}}\langle \partial_0^{-1}q_\kappa^{s}, q_\kappa^{z}\rangle
-1 \right| \ e^{-\kappa^2|t - z|^2/8}\ dz \longrightarrow 0,
\end{align*}
as $\kappa \rightarrow \infty$. This is because for any $\delta > 0$, one can choose a $N > 0$ large enough such that for $\kappa > N$, \beq \left|\frac{\kappa}{\sqrt{2\pi}}\langle \partial_0^{-1}q_\kappa^{s}, q_\kappa^{z}\rangle
-1 \right| < \delta, \nonumber \eeq for any $z \in (t-\epsilon, t + \epsilon)$. This follows from Item 1 in the lemma in the appendix found in \cite{EH-Lim02}. So the LHS of Equation (\ref{e.x.1}) converges to 0.
\end{proof}


\begin{lem}\label{l.v.1}
Let $\ol^u$ and $\ol^{\bar{u}}$ be 2 distinct open curves in $\bR \times \bR^3$ and let $\vec{y}^u,\ \vec{y}^{\bar{u}}: I \rightarrow R \subset \bR^3$ be parametrizations of $\ol^u$ and $\ol^{\bar{u}}$ respectively. Assume that the time components take on a definite sign, given by $\sgn(y_0^u)$ and $\sgn(y_0^{\bar{u}})$ respectively.

Project them into $\bR^3$ to form 2 distinct open curves $l^u$, $l^{\bar{u}}$ respectively, which when projected onto the plane $\Sigma_k$ gives a crossing which we will denote as $p_k$. Recall the algebraic number $\varepsilon(p_k)$ of a crossing $p_k$, given in Definition 2.9 (Algebraic crossing number) in \cite{CS-Lim02}. Note that the algebraic crossing for each of the $p_k$ is the same, which we will write as $\varepsilon$, for each $k=1, 2, 3$.

Let \beq \Lambda_\kappa =  \frac{\kappa}{4\pi}\left(\frac{\kappa}{\sqrt{2\pi}}\right)^2\left(\frac{\kappa^2}{8\pi} \right)^2. \nonumber \eeq
Let $\hat{s} = (s,\bar{s})$, $\hat{u} = (u, \bar{u})$ and write \beq \int_0^1 \int_0^1 ds d\bar{s} \ \epsilon^{ijk}y_{i,s}^{u,\prime}y_{j,\bar{s}}^{\bar{u},\prime}  \equiv \overline{\iint}_{\hat{s}, \hat{u}}^k . \nonumber \eeq

Let $R \subset \bR^3$ be a compact region which contains the curves $l^u$ and $l^{\bar{u}}$ and let $\rho: I^3 \rightarrow \bR^3$ be a parametrization for $R$. Refer to Section \ref{s.vo}. Then we have
\begin{align*}
\Lambda_\kappa&\int_{I^3} \Bigg[ \overline{\iint}_{\hat{s}, \hat{u}}^k\left\langle p_\kappa^{\vec{y}_{s}^u}, p_\kappa^{\vec{\rho}(r)} \right\rangle_k \left\langle \partial_0^{-1}q_\kappa^{y_{0,\bar{s}}^{\bar{u}}}, q_\kappa^{0} \right\rangle  e^{-\kappa^2|y_{\bar{s}}^{\bar{u}} - \rho(r)|^2/8} \Bigg]\ |J_\rho|(r)\ dr\\
& \hspace{1cm}\longrightarrow \sum_{k=1}^3\varepsilon(p_k)\sgn(y_0^u)\sgn(y_0^{\bar{u}}) = 3\varepsilon\ \sgn(y_0^u)\sgn(y_0^{\bar{u}}),
\end{align*}
as $\kappa \rightarrow \infty$.
\end{lem}

\begin{proof}

A direct computation will give \beq \left\langle p_\kappa^{\vec{y}_s^u}, p_\kappa^{(0,\rho(r))} \right\rangle_k = e^{-\kappa^2|\hat{y}_{k,s}^{u} - \hat{\rho}_k(r)|^2/8}\cdot \langle q_\kappa^{y_{k,s}^u}, \kappa \partial_0^{-1}q_\kappa^{\rho_k(r)}\rangle \cdot\langle \partial_0^{-1}q_\kappa^{y_{0,s}^u}, q_\kappa^{0}\rangle. \nonumber \eeq

Now, as $\kappa \rightarrow \infty$, Item 1 in the lemma in the appendix found in \cite{EH-Lim02} says that \beq \eta_\kappa(\hat{s}) := \frac{\kappa}{\sqrt{2\pi}}\left\langle \partial_0^{-1}q_\kappa^{y_{0,s}^u}, q_\kappa^{0} \right\rangle \cdot \frac{\kappa}{\sqrt{2\pi}}\left\langle \partial_0^{-1}q_\kappa^{y_{0,\bar{s}}^{\bar{u}}}, q_\kappa^{0} \right\rangle \rightarrow \sgn(y_0^u)\sgn(\varrho_0^{\bar{u}}). \label{e.v.1} \eeq

Let $R = \rho(I^3)$. Write \beq
r= (r_1, r_2, r_3),\ d\omega = dr_1dr_2dr_3,\ d\hat{s} = ds d\bar{s}\ {\rm and}\ \Delta_\kappa = \frac{\kappa}{2\sqrt{2\pi}}\left(\frac{\kappa^2}{8\pi} \right)^2, \nonumber \eeq
and \beq \frac{1}{\sqrt{2\pi}}\langle q_\kappa^{y_{k,s}^u}, \kappa\partial_0^{-1}q_\kappa^{a}\rangle =: \chi_{\kappa,s}^k(a),\ a \in \bR. \nonumber \eeq

Let $z \equiv (z_1, z_2, z_3) \in \bR^3$. We have
\begin{align*}
\Delta_\kappa &\sum_{k=1}^3\int_{I^3} \overline{\iint}_{\hat{s}, \hat{u}}^k e^{-\kappa^2|\hat{y}_{k,s}^{u} - \hat{\rho}_k(r)|^2/8}e^{-\kappa^2|y_{\bar{s}}^{\bar{u}} - \rho(r)|^2/8}\chi_{\kappa,s}^k(\rho_k(r))\ |J_\rho|(r)\ dr \\
=& \Delta_\kappa\sum_{k=1}^3 \int_{\bR^3}\overline{\iint}_{\hat{s}, \hat{u}}^ke^{-\kappa^2|\hat{y}_{k,s}^{u} - \hat{z}_k|^2/8}e^{-\kappa^2|y_{\bar{s}}^{\bar{u}} - z|^2/8}\chi_{\kappa,s}^k(z_k)\ dz  \\
&- \Delta_\kappa\sum_{k=1}^3 \int_{R^c}\overline{\iint}_{\hat{s}, \hat{u}}^ke^{-\kappa^2|\hat{y}_{k,s}^{u} - \hat{z}_k|^2/8}e^{-\kappa^2|y_{\bar{s}}^{\bar{u}} - z|^2/8}\chi_{\kappa,s}^k(z_k)\ dz  \\
=& \sum_{k=1}^3 \Delta_\kappa\int_{\bR^3}\overline{\iint}_{\hat{s}, \hat{u}}^ke^{-\kappa^2|\hat{y}_{k,s}^{u} - \hat{z}_k|^2/8} e^{-\kappa^2|\hat{y}_{k,\bar{s}}^{\bar{u}} - \hat{z}_k|^2/8}
e^{-\kappa^2|y_k^{\bar{u}} - z_k|/8}\chi_{\kappa,s}^k(z_k)
d\hat{z}_k dz_k \\
&- \sum_{k=1}^3\Delta_\kappa \int_{R^c}\overline{\iint}_{\hat{s}, \hat{u}}^k e^{-\kappa^2|\hat{y}_{k,s}^{u} - \hat{z}_k|^2/8}e^{-\kappa^2|y_{\bar{s}}^{\bar{u}} - z|^2/8}\chi_{\kappa,s}^k(z_k)\ dz  \\
:=& K_1(\kappa) + K_2(\kappa).
\end{align*}
It is straightforward to see that the second term $K_2(\kappa)$ converges to 0 as $\kappa$ goes to infinity, since $R^c$ does not contain any points in the curves $l^u$ and $l^{\bar{u}}$. See Expression \ref{e.vc.5}.

Note the following:
\begin{align}
\frac{\kappa}{2\sqrt{2\pi}}&\int_{\bR}\frac{\kappa}{\sqrt{2\pi}}\langle \partial_0^{-1}q_\kappa^{y_{k,s}^u}, q_\kappa^{z_k}\rangle e^{-\kappa^2|y_{k,\bar{s}}^{\bar{u}} - z_k|^2/8} dz_k =: \zeta_{\kappa}^k(\hat{s}) \longrightarrow \sgn( y_{k,\bar{s}}^{\bar{u}}- y_{k,s}^u ), \label{e.vc.2}\\
\left(\frac{\kappa^2}{8\pi} \right)^2&\int_{\bR^2}e^{-\kappa^2|\hat{y}_{k}^{u}(s) - \hat{z}_k|^2/8}e^{-\kappa^2|\hat{y}^{\bar{u}}(\bar{s}) - \hat{z}_k|^2/8}d\hat{z}_k = \frac{\kappa^2}{16\pi}e^{-\kappa^2|\hat{y}^{\bar{u}}(\bar{s}) - \hat{y}_{k}^{u}(s)|/16}. \label{e.vc.3}
\end{align}
The first equality follows from Lemma \ref{l.l.6} and the last equality follows from Equation (\ref{e.vc.1}), with $\theta = 8/\kappa^2$.

Together with Equations (\ref{e.vc.2}) and (\ref{e.vc.3}), and because $\langle \partial_0^{-1}f, g \rangle = -\langle f, \partial_0^{-1} g\rangle$, we see that
\beq K_1(\kappa) = \sum_{k=1}^3 \int_{I^2} \epsilon^{ijk}\frac{\kappa^2}{16\pi}e^{-\kappa^2|\hat{y}_k^{\bar{u}}(\bar{s}) - \hat{y}_{k}^{u}(s)|^2/16}[-\zeta_{\kappa}^k(\hat{s})]\ y_{i,s}^{u,\prime}y_{j,\bar{s}}^{\bar{u},\prime}\  d\hat{s}. \nonumber \eeq

So the limit is equivalent to compute the limit of
\begin{align}
\sum_{k=1}^3\int_{I^2}& \epsilon^{ijk}\frac{\kappa^2}{16\pi}e^{-\kappa^2|\hat{y}_{k,\bar{s}}^{\bar{u}} - \hat{y}_{k,s}^{u}|^2/16} [-\zeta_{\kappa}^k(\hat{s})]\eta_\kappa(\hat{s})\ y_{i,s}^{u,\prime}y_{j,\bar{s}}^{\bar{u},\prime}\  d\hat{s} . \label{e.vc.4}
\end{align}

Using a similar argument given in the lemma in the appendix found in \cite{EH-Lim02} and from Equations (\ref{e.v.1}) and (\ref{e.vc.2}), one can show that the above Expression \ref{e.vc.4} converges to \beq \sum_{k=1}^3\varepsilon(p_k)\sgn(y_0^u)\sgn(y_0^{\bar{u}}). \nonumber \eeq This completes the proof.
\end{proof}

\begin{rem}
Note that given a link $L \subset \bR^3$, we need to project it onto a plane to obtain a link diagram. It really does not matter which plane which choose to project onto. Note that we project the knot onto 3 planes $\Sigma_1, \Sigma_2$ and $\Sigma_3$. Now a set of crossings on a link diagram on a plane $\Sigma_1$ should be distinguished from a set of crossings on a link diagram on a plane $\Sigma_2$ or $\Sigma_3$. However, when we sum up the algebraic crossings, we should get the same value, regardless of the plane we choose. This accounts for the factor 3. See \cite{EH-Lim02}.
\end{rem}

By virtue of Lemma \ref{l.v.1}, we see that only relevant crossings formed from curves in $\bR^3$ will contribute to the limit.

For a fixed $u = 1, \ldots, \on$, the expression \beq \Lambda_\kappa\int_{I^3} \Bigg[ \overline{\iint}_{\hat{s}, \hat{u}}^k\left\langle p_\kappa^{\vec{y}_{s}^u}, p_\kappa^{\vec{\rho}(r)} \right\rangle_k \left\langle \partial_0^{-1}q_\kappa^{y_{0,\bar{s}}^u}, q_\kappa^{0} \right\rangle  e^{-\kappa^2|y_{\bar{s}}^{u} - \rho(r)|^2/8} \Bigg]\ |J_\rho|(r)\ dr,\ \hat{u} \equiv (u, u) \label{e.v.2} \eeq does not have a limit. This is because Equation (\ref{e.vc.2}) and Expression \ref{e.vc.4} are not defined. This problem is termed as the self-linking problem for a knot in $\bR^3$.

The solution would be to consider a frame $v^u \in \bR^3 \cong \{0\} \times \bR^3$ on the knot $\pi_0(\ol^u)$ and define a new loop $\ol^{u, \epsilon} = \ol^u + \epsilon v^u$, i.e. we consider a displaced copy of $\ol^u$, call it $\ol^{u, \epsilon}$ and project both $\ol^u$ and $\ol^{u, \epsilon}$ onto $\Sigma_k \subset \bR^3$ to form a link diagram as defined in Definition 2.6 in \cite{CS-Lim02}. This will give us the self-linking number of $\pi_0(\ol^u)$. See \cite{CS-Lim02} for a detailed description of the self-linking number.

Let $\vec{y}^{u, \epsilon}$ be a parametrization of $\ol^{u, \epsilon}$. We now define the limit of Expression as \ref{e.v.2} as \beq \lim_{\epsilon \rightarrow 0}\lim_{\kappa \rightarrow 0}\Lambda_\kappa\int_{I^3} \Bigg[ \overline{\iint}_{\hat{s}, \hat{u}}^k\left\langle p_\kappa^{\vec{y}_{s}^u}, p_\kappa^{\vec{\rho}(r)} \right\rangle_k \left\langle \partial_0^{-1}q_\kappa^{y_{0,\bar{s}}^{u,\epsilon}}, q_\kappa^{0} \right\rangle  e^{-\kappa^2|y_{\bar{s}}^{u,\epsilon} - \rho(r)|^2/8} \Bigg]\ |J_\rho|(r)\ dr. \label{e.v.3} \eeq

From the proof of Lemma \ref{l.v.1}, one can show that \beq
\lim_{\kappa \rightarrow 0}\frac{\kappa}{4\pi}\left( \frac{\kappa^2}{8\pi}\right)^2\int_{I^3} \Bigg[ \overline{\iint}_{\hat{s}, \hat{u}}^k\left\langle p_\kappa^{\hat{y}_{k,s}^u}, p_\kappa^{\hat{\rho}_k(r)} \right\rangle \cdot \left\langle q_\kappa^{y_{k,s}^{u}}, \kappa \partial_0^{-1}q_\kappa^{\rho_k(r)} \right\rangle   e^{-\kappa^2|y_{\bar{s}}^{u,\epsilon} - \rho(r)|^2/8} \Bigg]\ |J_\rho|(r)\ dr \nonumber \eeq
is equal to the sum of the algebraic crossing number of the crossings formed between $y^u$ and $y^{u,\epsilon}$.

If we further assume that $R$ is so small that it fits inside a small tubular neighborhood of $y^u$ and contains segments of $y^u$, then one can show that when we take the limit as $\epsilon$ goes to 0, these crossings coincide, and were termed as half-twists in \cite{CS-Lim02}. We refer the reader to \cite{CS-Lim02} for the details. We will state this result as a corollary of Lemma \ref{l.v.1}.

\begin{defn}\label{d.vx.1}
Let $\tilde{\pi}_3: \bR^3 \rightarrow \Sigma_3$. Given a compact region $R$ and a framed knot $l$, project $R$ to be a planar set $\tilde{\pi}_3(R) \subset \Sigma_3$ and $l$ to be a graph as defined in \cite{CS-Lim02}, such that each vertex has valency 2 or 4. We define ${\rm TDP}(l; R)$ to be the set of all half-twists from the framed knot $l$, i.e. vertices with valency 2 which are in the interior of $\tilde{\pi}_3(R)$.
\end{defn}

\begin{rem}
We can define the set of all half-twists using other planes. If we project it on a different plane, the set of half-twists will of course be different. But as our final result will only depend on the number of half-twists, so it really does not matter which plane we project onto.
\end{rem}

\begin{cor}\label{c.v.2}
Recall the algebraic number $\varepsilon(p)$ of a half-twist $p$. The limit of Expression \ref{e.v.2} is defined by Expression \ref{e.v.3}, which is equal to \beq  3\sum_{p \in {\rm TDP}(\pi_0(\ol^u); R)} \varepsilon(p). \nonumber \eeq Here, ${\rm TDP}(\pi_0(\ol^u); R)$ refers to the set of half-twists from the knot $\pi_0(\ol^u)$, as defined in Definition \ref{d.vx.1}.
\end{cor}

The limit as $\kappa$ goes to infinity, will give us the following result.

\begin{lem}\label{l.v.2}
Refer to Notation \ref{n.r.2}. Let $R_v:= \rho(I_v^3)$. Recall ${\rm TDP}(\pi_0(\ol^u); R_v)$ refers to the set of half-twists from the knot $\pi_0(\ol^u)$, which are in the interior of $\tilde{\pi}_3(R_v)$ and let $\left| {\rm TDP}(\pi_0(\ol^u); R_v) \right|$ be the total number of half-twists in the set.

Let \beq \Lambda_\kappa =  \frac{\kappa}{4\pi}\left(\frac{\kappa}{\sqrt{2\pi}}\right)^2\left(\frac{\kappa^2}{8\pi} \right)^2. \nonumber \eeq Let $\hat{s} = (s,\bar{s})$ and write
\beq \sum_{u=1}^{\on}\int_0^1 \int_0^1ds d\bar{s}\ \epsilon^{ijk}y_{i,s}^{u,\prime}y_{j,\bar{s}}^{u,\prime}   \equiv \overline{\iint}_{\hat{s}, u}^k. \nonumber \eeq

Then
\begin{align*}
&\lim_{\kappa \rightarrow \infty}\Lambda_\kappa\int_{I_v^3} \Bigg|\overline{\iint}_{\hat{s}, u}^k  \left\langle p_\kappa^{\vec{y}_{s}^u}, p_\kappa^{\vec{\rho}(r)} \right\rangle_k \left\langle \partial_0^{-1}q_\kappa^{y_{0,\bar{s}}^{u}}, q_\kappa^{0} \right\rangle e^{-\kappa^2|y_{\bar{s}}^{u} - \rho(r)|^2/8}  \ \frac{\xi_{\rho_{u}^\pm}}{3}\Bigg|\ |J_\rho|(r)\ dr\\
& :=\lim_{\epsilon \rightarrow 0}\lim_{\kappa \rightarrow 0}\Lambda_\kappa\int_{I_v^3}\Bigg|  \overline{\iint}_{\hat{s}, u}^k\left\langle p_\kappa^{\vec{y}_{s}^u}, p_\kappa^{\vec{\rho}(r)} \right\rangle_k \left\langle \partial_0^{-1}q_\kappa^{y_{0,\bar{s}}^{u,\epsilon}}, q_\kappa^{0} \right\rangle  e^{-\kappa^2|y_{\bar{s}}^{u,\epsilon} - \rho(r)|^2/8} \ \frac{\xi_{\rho_{u}^\pm}}{3} \Bigg|\ |J_\rho|(r)\ dr \\
&= \left|{\rm TDP}(\pi_0(\ol^u); R_v)  \right|\xi_{\rho_u^\pm}.
\end{align*}
\end{lem}

\begin{proof}
If $\rho(I_v^3)$ is not inside a tubular neighborhood of some knot $l^u := \pi_0(\ol^u)$, then it is
straightforward to show that
\begin{align*}
\Lambda_\kappa &\int_{I_v^3} \Bigg|  \int_0^1 \int_0^1ds d\bar{s} \ \epsilon^{ijk}y_{i,s}^{u,\prime}y_{j,\bar{s}}^{u,\prime} \left\langle p_\kappa^{\vec{y}_{s}^u}, p_\kappa^{\vec{\rho}(r)} \right\rangle_k \left\langle \partial_0^{-1}q_\kappa^{y_{0,\bar{s}}^{u,\epsilon}}, q_\kappa^{0} \right\rangle e^{-\kappa^2|y_{\bar{s}}^{u,\epsilon} - \rho(r)|^2/8}  \ \frac{\xi_{\rho_{u}^\pm}}{3}\Bigg| \ |J_\rho|(r)\ dr \\
&\longrightarrow 0 ,
\end{align*}
as $\kappa$ goes to infinity.

If $R_v$ is contained inside the tubular neighborhood of some knot $l^u$, then Corollary \ref{c.v.2} says that
\begin{align*}
\Lambda_\kappa& \int_{I_v^3}\Bigg|ds d\bar{s}\  \int_0^1 \int_0^1 \epsilon^{ijk}y_{i,s}^{u,\prime}y_{j,\bar{s}}^{u,\prime}  \left\langle p_\kappa^{\vec{y}_{s}^u}, p_\kappa^{\vec{\rho}(r)} \right\rangle_k \left\langle \partial_0^{-1}q_\kappa^{y_{0,\bar{s}}^{u,\epsilon}}, q_\kappa^{0} \right\rangle e^{-\kappa^2|y_{\bar{s}}^{u,\epsilon} - \rho(r)|^2/8}  \ \frac{\xi_{\rho_{u}^\pm}}{3}\Bigg|\ |J_\rho|(r)\ dr \\
& \longrightarrow \left|{\rm TDP}(\pi_0(\ol^u); R_v)  \right|\xi_{\rho_u^\pm}.
\end{align*}

\end{proof}

We can state our main theorem.

\begin{thm}(Main Theorem)\label{t.main.3}\\
Consider two oriented hyperlinks, $\oL = \{\ol^u \}_{u=1}^{\on}$, $\uL = \{\ul^v \}_{v=1}^{\un}$ in $\bR \times \bR^3$ with non-intersecting (closed) loops, the former colored with a representation $\rho_u \equiv (\rho^+_u, \rho^-_u): \mathfrak{su}(2)\times \mathfrak{su}(2) \rightarrow {\rm End}(V_u^+) \times {\rm End}(V_u^-)$. These two oriented hyperlinks together, form a new colored oriented hyperlink, denoted by $\chi(\oL, \uL)$. Let $R \subset \bR^3 \cong \{0\} \times \bR^3$ be a compact region and disjoint from $\oL$.

Project $\ol^u$ in $\bR^3$ to form a knot, denoted by $\pi_0(\ol^u)$. Refer to Definition \ref{d.vx.1} for the definition of ${\rm TDP}(\pi_0(\ol^u); R)$, the set of half-twists from the knot $\pi_0(\ol^u)$ which are inside the interior of a planar set $\tilde{\pi}_3(R)$. And $\left|{\rm TDP}(\pi_0(\ol^u); R) \right|$ refers to the total number of half-twists inside the set.

Let $\hat{V}_R$ be the volume operator corresponding to $R$. Then $\hat{V}_R$ acts on the Wilson Loop observable for a colored hyperlink $\chi(\oL, \uL)$, $Z(q; \chi(\oL, \uL))$ via the path integral Expression \ref{e.v.11}, which is defined as the limit of Expression \ref{ex.v.3} as $\kappa$ goes to infinity, given by
\begin{align*}
\hat{V}_R[Z(q; \chi(\oL, \uL))] := \frac{q^2\pi^{3/2}}{2}\prod_{\bar{u}=1}^{\on}\Bigg\{& \left[ \sum_{u=1}^{\on} \left|{\rm TDP}(\pi_0(\ol^u); R) \right|
\xi_{\rho_u^+}   \right]^{1/\on}
\Tr_{\rho^+_{\bar{u}}}\ \exp[-\pi iq\ {\rm sk}(\ol^{\bar{u}}, \uL) \cdot \mathcal{E}^+] \\
+& \left[\sum_{u=1}^{\on} \left| {\rm TDP}(\pi_0(\ol^u); R)\right|
\xi_{\rho_u^-}   \right]^{1/\on}
\Tr_{\rho^-_{\bar{u}}}\ \exp[\pi iq\ {\rm sk}(\ol^{\bar{u}}, \uL) \cdot \mathcal{E}^-]
\ \ \Bigg\}.
\end{align*}
\end{thm}

\begin{proof}
The proof follows from Equation (\ref{e.w.5}) and Lemma \ref{l.v.2}.
\end{proof}

\begin{rem}
Note that the volume functional $V_R$ commutes with the holonomy operator. As such, we expect and require that the quantized operator $\hat{V}_R$ be proportional to the identity.
\end{rem}

\section{Application to Quantum Field Theory}

There is some inconsistency in Quantum Field Theory, as explained by Thiemann in \cite{Thiemann:2002nj}. Because of the divergences one encounter in the computations in Quantum Field Theory,  Thiemann described it as an incomplete theory and called for a new theory to rectify this problem.

At short distances, in the order of Planck's distance (around $10^{-33}$ cm), the current Quantum Field Theory predicts the existence of virtual particles, which has large momentum $p$ and energy $E$. Its Compton length is inversely proportional to $p$ and its Schwarzschild radius is proportional to $E$. As its momentum $p$ increases and hence $E$ increases, we see that its Compton length and its Schwarzschild radius decreases and increases respectively. When its Compton length is equal to its Schwarzschild radius, General Relativity predicts that this particle will turn into a black hole. When this happens, Hawking radiation and all sorts of particles will henceforth be emitted.

At such short distance and high energy, quantum gravity should come into play. Now, we showed that in quantum gravity, Theorem \ref{t.main.3} says that volume is discretized, hence implying that length has to be discretized. Furthermore, it also implies that the energy of the particle has to be discretized. Therefore, at Planck's distance, Quantum Field Theory and General Relativity should no longer apply in this regime, so the above qualitative picture does not apply. In other words, before a particle's Compton length reaches Schwarzschild radius, classical General Relativity is no longer valid and Quantum Gravity takes over, preventing a quantum black hole from forming.

The fact that volume is discretized is the very essence of the discreteness of space-time in short distance and high energy, hence give birth to the idea of fundamental discreteness in quantum geometry. See \cite{0264-9381-21-15-R01}, \cite{rovelli2004quantum}, \cite{Rovelli1998} and \cite{Thiemann:2007zz}.



\end{document}